\documentclass{article}
\usepackage[utf8]{inputenc}
\usepackage[a4paper]{geometry}
\usepackage{amsmath}
\usepackage{amsthm}
\usepackage{color}
\usepackage{amssymb}
\usepackage{graphicx}

\newtheorem{theorem}{Theorem}[section]
\newtheorem{remark}{Remark}[section]
\newtheorem{corollary}{Corollary}[theorem]

\theoremstyle{definition}
\newtheorem{definition}{Definition}[section]
\newtheorem{example}{Example}[section]

\title{On the Topology $\tau^{\diamond}_R$ of Primal Topological Spaces}
\author{Murad ÖZKOÇ$^1$, Büşra KÖSTEL$^2$\\
$^1$Mu\u{g}la S{\i}tk{\i} Ko\c{c}man University, 
Faculty of Science, \\Department of Mathematics  48000,
Mente\c{s}e-Mu\u{g}la, Turkey\\
$^2$Mu\u{g}la S{\i}tk{\i} Ko\c{c}man University,\\ 
Graduate School of Natural and Applied Sciences, \\Mathematics, 48000, Mente\c{s}e-Mu\u{g}la, Turkey\\
$e$-mails: $^1$murad.ozkoc@mu.edu.tr\\
$^2$buslakostel@gmail.com}

\begin{document}
\maketitle
\begin{abstract}
The main purpose of this paper is to introduce and study two new operators $(\cdot)_R^{\diamond}$ and $cl_R^{\diamond}(\cdot)$ via primal which is a new notion. We also show that the operator $cl_R^{\diamond}(\cdot)$ is a Kuratowski closure operator, while the operator $(\cdot)_R^{\diamond}$ is not.
In addition, we prove that the topology on $X$, shown as $\tau_R^{\diamond},$ obtained by means of the operator $cl_R^{\diamond}(\cdot)$ is finer than $\tau_{\delta},$ where $\tau_{\delta}$ is the family of $\delta$-open subsets of a space $(X,\tau).$ Moreover, we not only obtain a base for the topology $\tau_R^{\diamond}$ but also prove many fundamental results concerning this new structure. Furthermore, we give many counterexamples related to our results. \\  
{\bf 2020 AMS Classifications:} 54A05;  54B99; 54C60.

\end{abstract}

\section{Introduction}
The desire to obtain valid solutions to many issues in topology, such as compactification, proximity space and closure space problems, has led to the introduction of some new structures by topologists. Some of these structures are nets, filters \cite{Ku}, ideals \cite{Ja} and grills \cite{ch}. The classical structures such as nets, filters, ideals, and grills are undoubtedly some of the most important objects of topology. Kuratowski defined and studied the idea of ideal from filter \cite{Ku}. Also, this notion has been studied by many topologists in different directions in \cite{m1, m2}. The notion of ideal comes across as the dual structure of filter. Also, the other classical structure of general topology in the literature is the notion of grill.  It was introduced by Choquet \cite{ch} in 1947 and studied by many authors in \cite{wj, b1, b2, b3, b4, m1, m2, ro, Na}. Recently, Acharjee et al. introduced a new classical structure called primal \cite{aoi}. They define the notion of primal topological space by utilizing two new operators and investigate many fundamental properties of this new structure and these two operators. Moreover, the notion of primal is the dual structure of grill. Furthermore, some new studies have been revealed regarding primal topological spaces in \cite{ ahm, hanan, ahmad, ahmadandmesfer} since the introduction of primal.
\section{Preliminaries}
Throughout this paper, $(X,\tau)$ and $(Y,\sigma)$ (briefly, $X$ and $Y$) represent topological spaces unless otherwise stated. We denote the closure and interior of a subset $A$ of a space $X$ by $cl(A)$ and $int(A)$, respectively. Also, the powerset of a set $X$ will be denoted by $2^X.$ The family of all open neighborhoods of a point $x$ of $X$ will be denoted by $O(X,x).$ \\

A subset $A$ of a space $X$ is called regular open \cite{sto} if $A=int(cl(A)).$ The complement of a regular open set is called regular closed \cite{sto}. The union of all regular open subsets of $X$ contained in $A$ is called $\delta$-interior of $A$ and denoted by $\delta\text{-}int(A).$ Dually, the intersection of all regular closed subsets of $X$ containing $A$ of a space $X$ is called $\delta$-closure of $A$ and denoted by $\delta\text{-}cl(A).$ A subset $A$ of a space $X$ is called $\delta$-open \cite{vel} if $A=\delta\text{-}int(A).$ The complement of a $\delta$-open set in a space $X$ is called $\delta$-closed \cite{vel}. The family of all regular open \cite{sto} (resp. regular closed \cite{sto}, $\delta$-open \cite{vel}, $\delta$-closed \cite{vel}) subsets of a space $X$ will be denoted by $RO(X)$ $(\text{resp. } RC(X), \ \delta O(X), \ \delta C(X)).$ The family of all regular open (resp. regular closed, $\delta$-open, $\delta$-closed) sets of $X$ containing a point $x$ of $X$ is denoted by $RO(X,x)$ (resp. $RC(X,x)$, $\delta O(X,x)$, $\delta C(X,x)).$ 
\\

Now, we recall the following results concerning $\delta$-interior and $\delta$-closure of a set $A$ in a topological space $X$.
\\
\begin{theorem} \cite{vel}
Let $X$ be a topological space and $A\subseteq X.$ Then the followings hold.\\

$(a)$ $\delta\text{-}int(A)=\{x|(\exists U\in RO(X,x))(U\subseteq A)\}=\{x|(\exists U\in O(X,x))(int(cl(U))\subseteq A)\},$
\\

$(b)$ $\delta\text{-}cl(A)=\{x|(\forall U\in RO(X,x))(U\cap A\neq\emptyset)\}=\{x|(\forall U\in O(X,x))(int(cl(U))\cap A\neq\emptyset)\},$\\

$(c)$ $\delta\text{-}cl(A^c)=(\delta\text{-}int(A))^c,$\\

$(d)$ $\delta\text{-}int(A^c)=(\delta\text{-}cl(A))^c.$
\end{theorem}






\begin{definition} \cite{aoi}
Let $X$ be a non-emptyset. A collection $\mathcal{P}\subseteq   2^X$  is called a primal on $X$ if it satisfies the following conditions:\\

$a)$ $X \notin \mathcal{P}$,\\

$b)$ if $A\in \mathcal{P}$ and $B\subseteq A$, then $B\in \mathcal{P}$,\\

$c)$ if $A\cap B\in \mathcal{P}$, then $A\in \mathcal{P}$ or $B\in \mathcal{P}$.\\
\end{definition}

\begin{corollary} \cite{aoi}
Let $X$ be a non-emptyset. A collection  $\mathcal{P}\subseteq   2^X$  is a primal on $X$ if and only if it satisfies the following conditions:\\  

$a)$ $X \notin \mathcal{P}$,\\

$b)$ if $B\notin \mathcal{P}$ and $B\subseteq A$, then $A\notin \mathcal{P}$,\\

$c)$ if $A\notin \mathcal{P}$ and $B\notin \mathcal{P},$ then $A\cap B\notin \mathcal{P}.$\\
\end{corollary}

\begin{definition} \cite{aoi}
A topological space $(X,\tau)$ with a primal $\mathcal{P}$ on $X$ is called a primal topological space and denoted by $(X,\tau,\mathcal{P}).$
\end{definition}

\begin{definition} \cite{aoi}
Let $(X,\tau,\mathcal{P})$ be a primal topological space. We consider a map $(\cdot)^{\diamond}: 2^X \rightarrow 2^X$ as $A^{\diamond}(X,\tau,\mathcal{P})=\{x\in X | (\forall U\in O(X,x))(A^c\cup U^c \in \mathcal{P})\}$ for any subset $A$ of $X.$ 
We can also write $A^{\diamond}$ as $A^{\diamond}(X,\tau,\mathcal{P})$ to specify the primal as per our requirements.  
\end{definition}

\section{The $(\cdot)^{\diamond}_R$ operator and its fundamental properties}

In this section, we introduce and study a new operator in primal topological spaces. We obtain not only some fundamental properties of this new operator but also give many counterexamples related to this operator. Now, we have the following definition of the  operator $(\cdot)^{\diamond}_R$.\\

\begin{definition}
Let $(X,\tau,\mathcal{P})$ be a primal topological space. We define a map $(\cdot)^{\diamond}_R: 2^X \rightarrow 2^X$ as $A^{\diamond}_R=\{x\in X : (\forall U\in RO(X,x))(A^c\cup U^c \in \mathcal{P})\}$ for any subset $A$ of $X.$ 
We can also use the notation $A_R^{\diamond}$ as $A^{\diamond}_R(X,\tau,\mathcal{P})$ to indicate the primal and the topology as per our requirements.  
\end{definition}

\begin{corollary} \label{15}
Since the family of all open sets and the family of all regular open sets are the same in regular topological spaces, the sets $A_R^{\diamond}$ and $A^{\diamond}$ coincide in that spaces.
\end{corollary}

\begin{corollary} \label{16}
Let $X$ be a topological space and $A\subseteq X.$ If $X$ is a compact Hausdorff space, then the sets $A_R^{\diamond}$ and $A^{\diamond}$ coincide.
\end{corollary}

\begin{proof}
 It is obvious since every compact Hausdorf space is regular.    
\end{proof}




\begin{remark}
Let $(X,\tau,\mathcal{P})$ be a primal topological space and $A\subseteq X.$ The inclusions of $A^{\diamond}_R\subseteq A$ or $A\subseteq A^{\diamond}_R$ need not always be true as shown by the following examples.
\end{remark}

\begin{example}
Let $X=\{a,b,c\},$ $\tau=\{\emptyset,\{b\}, \{c\}, \{b,c\},\{a,c\}, X\}$ and $\mathcal{P} = \{\emptyset, \{b\}, \{c\}, \{b,c\}\}.$ For the subset $A=\{b,c\},$ we get $A^{\diamond}_R=\emptyset$ but $A=\{b,c\}\nsubseteq \emptyset=A^{\diamond}_R.$
\end{example}

\begin{example}
Let $X=\{a,b,c\},$ $\tau=\{\emptyset,X\}$ and $\mathcal{P}=\{\emptyset,\{b\}, \{c\},\{b,c\}\}.$ For the subset $A=\{a\},$ we get  $A^{\diamond}_R=X$ but $A^{\diamond}_R=X\nsubseteq \{a\}=A.$
\end{example}

\begin{theorem} \label{1}
Let $(X,\tau,\mathcal{P})$ be a primal topological space and $A,B\subseteq X$. Then, the following statements hold.\\

$a)$ $A^{\diamond}\subseteq A^{\diamond}_R,$
\\

$b)$ If $A\in \delta C(X),$ then $ A^{\diamond}_R\subseteq A,$\\

$c)$ $\emptyset^{\diamond}_R= \emptyset$,\\

$d)$ $A^{\diamond}_R\in \delta C(X),$ \\

$e)$ $(A^{\diamond}_R)^{\diamond}_R\subseteq A^{\diamond}_R,$ \\

$f)$ If $A\subseteq B$, then $A^{\diamond}_R\subseteq B^{\diamond}_R$,\\

$g)$ $A^{\diamond}_R\cup B^{\diamond}_R= (A\cup B)^{\diamond}_R,$\\


$h)$ $ (A\cap B)^{\diamond}_R\subseteq A^{\diamond}_R\cap B^{\diamond}_R.$\\
\end{theorem}

\begin{proof}
$a)$ It is obvious since every regular open set in topological spaces is open.
\\

$b)$ Let $A\in \delta C(X)$ and $x\notin A.$ Our aim is to show that $x\notin A_R^{\diamond}.$
\\
$\left.\begin{array}{r}
A\in \delta C(X)\Rightarrow A^c\in\delta O(X)\Rightarrow (\exists\mathcal{A}\subseteq RO(X))\left(A^c=\bigcup \mathcal{A}\right) \\x\notin A\Rightarrow x\in A^c  
\end{array}\right\}\Rightarrow (\exists B\in RO(X))(x\in B\subseteq A^c)$
\\
$\begin{array}{l}\Rightarrow (B\in RO(X,x))(B^c\cup A^c\supseteq (A^c)^c\cup A^c=A\cup A^c=X\notin\mathcal{P})\end{array}$
\\
$\begin{array}{l}\Rightarrow (B\in RO(X,x))(B^c\cup A^c\notin\mathcal{P})\end{array}$
\\
$\begin{array}{l}\Rightarrow x\notin A_R^{\diamond}.\end{array}$
\\

$c)$ $\emptyset_R^{\diamond}=\{x\in X:(\forall U\in RO(X,x))(U^c\cup \emptyset^c=U^c\cup X=X\in\mathcal{P})\}=\emptyset.$
\\


$d)$ Our aim is to show that $A^{\diamond}_R=\delta\text{-}cl(A^{\diamond}_R).$ We have always $A^{\diamond}_R\subseteq \delta\text{-}cl(A^{\diamond}_R)\ldots (1)$
\\

Conversely, now let $x\in \delta\text{-}cl(A^{\diamond}_R)$ and $U\in RO(X,x).$
\\
$
\left.\begin{array}{r}
U\in RO(X,x)\\ x\in \delta\text{-}cl(A^{\diamond}_R)
\end{array}\right\}\Rightarrow U\cap A_R^{\diamond}\neq \emptyset\Rightarrow (\exists y\in X)(y\in U\cap A_R^{\diamond})\Rightarrow (\exists y\in X)(y\in U)(y\in A_R^{\diamond})
$
\\
$
\begin{array}{l}
\Rightarrow (U\in RO(X,y))(y\in A_R^{\diamond})
\end{array}
$
\\
$
\begin{array}{l}
\Rightarrow U^c\cup A^c\in\mathcal{P}
\end{array}
$

Then, we have $x\in A_R^{\diamond}.$ Thus,  $\delta\text{-}cl(A^{\diamond}_R)\subseteq A^{\diamond}_R\ldots (2)$
\\

$
\begin{array}{l}
(1),(2)\Rightarrow \delta\text{-}cl(A^{\diamond}_R)= A^{\diamond}_R\Rightarrow A^{\diamond}_R\in \delta C(X).
\end{array}
$
\\

$e)$ It is obvious from $(b)$ and $(d).$\\

$f)$ Let $A\subseteq B$ and $x\in A^{\diamond}_R.$ Our aim is to show that $x\in B^{\diamond}_R.$
\\
$
\left.\begin{array}{r}
x\in A^{\diamond}_R\Rightarrow (\forall U\in RO(X,x))(U^c\cup A^c\in\mathcal{P})
\\
A\subseteq B\Rightarrow B^c\subseteq A^c
\end{array}\right\}\Rightarrow 
$
\\
$\left.\begin{array}{rr}
\Rightarrow (\forall U\in  RO(X,x))(U^c\cup B^c\subseteq U^c\cup A^c\in\mathcal{P})  \\
\mathcal{P} \text{ is a primal on }  X
\end{array}\right\}\Rightarrow  (\forall U\in  RO(X,x))(U^c\cup B^c\in\mathcal{P})\Rightarrow x\in B^{\diamond}_R.$
\\

$g)$ Let $A,B\subseteq X.$
\\
$\left.\begin{array}{r}
A,B\subseteq X\Rightarrow A\subseteq A\cup B\Rightarrow A_R^{\diamond}\subseteq (A\cup B)_R^{\diamond} \\
A,B\subseteq X\Rightarrow B\subseteq A\cup B\Rightarrow B^{\diamond}_R\subseteq (A\cup B)_R^{\diamond}
\end{array}\right\}\Rightarrow A^{\diamond}_R\cup B^{\diamond}_R\subseteq (A\cup B)_R^{\diamond}\ldots (1)$
\\

Conversely, let $x\notin A^{\diamond}_R\cup B^{\diamond}_R.$
\\
$\left.\begin{array}{r}
x\notin A^{\diamond}_R\cup B^{\diamond}_R\Rightarrow (x\notin A^{\diamond}_R)(x\notin B^{\diamond}_R)\Rightarrow (\exists U,V\in RO(X,x))(U^c\cup A^c\notin\mathcal{P})(V^c\cup B^c\notin\mathcal{P}) \\
W:=U\cap V
\end{array}\right\}\Rightarrow $
\\
$
\begin{array}{l}
\Rightarrow (W\in  RO(X,x))(W^c\cup A^c\notin\mathcal{P})(W^c\cup B^c\notin\mathcal{P})
\end{array}
$
\\
$
\begin{array}{l}
\Rightarrow (W\in RO(X,x))(W^c\cup (A\cup B)^c=(W^c\cup A^c)\cap (W^c\cup B^c)\notin\mathcal{P})
\end{array}
$
\\
$
\begin{array}{l}
\Rightarrow x\notin (A\cup B)_R^{\diamond}
\end{array}
$

Then we have $(A\cup B)_R^{\diamond}\subseteq A_R^{\diamond}\cup B_R^{\diamond}\ldots (2)$
\\

$
\begin{array}{l}
(1),(2)\Rightarrow (A\cup B)_R^{\diamond}=A_R^{\diamond}\cup B_R^{\diamond}.
\end{array}
$
\\

$h)$ It is clear from $(f).$
\end{proof}

\begin{remark}
The converse of Theorem \ref{1}(h) need not always be true as shown by the following example.
\end{remark}


\begin{example}
Let $X=\{a,b,c\},$  $\tau=\{\emptyset,\{a\},\{c\},\{a,c\},X\}$ and $\mathcal{P} = \{\emptyset, \{a\}, \{b\},\{c\},\{a,b\},\{a,c\}\}.$ Now, if $A=\{b\}$ and $B=\{c\},$ then we have $A_R^{\diamond}\cap B_R^{\diamond}=\{b\}\cap \{b,c\}=\{b\} \neq \emptyset =\emptyset_R^{\diamond}=(A\cap B)_R^{\diamond}.$
\end{example}


\begin{theorem} \label{mt}
Let $(X,\tau,\mathcal{P})$ be a primal topological space and $A,B\subseteq X.$ If $A\in \delta O(X),$ then $A\cap B^{\diamond}_R\subseteq (A\cap B)^{\diamond}_R.$
\end{theorem}

\begin{proof}
Let $A\in \delta O(X)$ and $x\in A\cap B^{\diamond}_R.$
\\
$\left.\begin{array}{r}
x\in A\cap B^{\diamond}_R\Rightarrow (x\in A)(x\in B^{\diamond}_R)\Rightarrow (x\in A)(\forall U\in RO(X,x))(U^c\cup B^c\in\mathcal{P})  \\
 A\in \delta O(X)\Rightarrow (\exists\mathcal{A}\subseteq RO(X))\left(A=\bigcup\mathcal{A}\right)\Rightarrow (\exists V\in RO(X))(V\subseteq A)
\end{array}\right\}\Rightarrow $
\\
$
\begin{array}{l}
\Rightarrow (\forall U\in RO(X,x))(U\cap V\in RO(X,x))(U^c\cup (A\cap B)^c=(U\cap A)^c\cup B^c\subseteq (U\cap V)^c\cup B^c\in\mathcal{P})
\end{array}
$
\\
$
\begin{array}{l}
\Rightarrow (\forall U\in RO(X,x))(U^c\cup (A\cap B)^c\in\mathcal{P})
\end{array}
$
\\
$
\begin{array}{l}
\Rightarrow x\in (A\cap B)_R^{\diamond}.
\end{array}
$
\end{proof}

\begin{theorem}
Let $(X,\tau,\mathcal{P})$ be a primal topological space. Then, the following statements are equivalent:
\\

$a)$ $X_R^{\diamond}=X;$
\\

$b)$ $RC(X)\setminus \{X\}\subseteq \mathcal{P};$  
\\

$c)$ $A\subseteq A_R^{\diamond}$ for all regular open subsets $A$ of $X.$ 
\end{theorem}

\begin{proof}
$(a)\Rightarrow (b):$ Let $X_R^{\diamond}=X.$
$$\begin{array}{rcl} X_R^{\diamond}=X & \Rightarrow & (\forall x\in X)(x\in X_R^{\diamond}) \\ \\ & \Rightarrow & (\forall x\in X)(\forall U\in RO(X,x))(U^c\cup X^c=U^c\in\mathcal{P}) \\ \\ & \Rightarrow & (\forall V\in RC(X)\setminus\{X\})(V\in\mathcal{P}) \\ \\ & 
 \Rightarrow & RC(X)\setminus \{X\}\subseteq \mathcal{P}. \end{array}$$

$(b)\Rightarrow (a):$ Let $x\in X$ and $U\in RO(X,x).$ 
\\
$\left.\begin{array}{r}
U\in RO(X,x)\Rightarrow U^c\in RC(X)\setminus \{X\}\\
\text{Hypothesis}
\end{array}\right\}\Rightarrow U^c\cup X^c=U^c\cup\emptyset=U^c\in\mathcal{P}$
\\

Then, we have $x\in X_R^{\diamond}.$ Thus, $X\subseteq X_R^{\diamond}\subseteq X$ and so $X_R^{\diamond}=X.$
\\

$(b)\Rightarrow (c):$ Let $A\in RO(X).$
\\
$
\left.\begin{array}{r}
A\in RO(X)\overset{\text{Theorem }    \ref{mt}}{\Rightarrow} A\cap X_R^{\diamond}\subseteq (A\cap X)_R^{\diamond}=A_R^{\diamond}
\\
RC(X)\setminus \{X\}\subseteq \mathcal{P}\Rightarrow X_R^{\diamond}=X
\end{array}\right\}\Rightarrow A\subseteq A_R^{\diamond}.
$
\\

$(c)\Rightarrow (b):$ Let $A\in RC(X)\setminus \{X\}.$ 
\\
$\left.\begin{array}{r}
A\in RC(X)\setminus \{X\}\Rightarrow  A^c\in RO(X)\setminus \{\emptyset\}\\
\text{Hypothesis}\end{array}\right\}\Rightarrow A^c\subseteq (A^c)_R^{\diamond}\Rightarrow (\forall x\in A^c)(x\in (A^c)_R^{\diamond})$
\\
$\left.\begin{array}{rr}\Rightarrow (\forall x\in A^c)(\forall U\in RO(X,x))(U^c\cup (A^c)^c=U^c\cup A\in\mathcal{P}) \\ A\subseteq U^c\cup A\end{array}\right\}\Rightarrow A\in\mathcal{P}$
\\

Then, we have $RC(X)\setminus \{X\}\subseteq \mathcal{P}.$
\end{proof}

\begin{theorem}\label{5}
Let $\mathcal{P}$ be a primal on topological space $(X,\tau)$ and $A\subseteq X.$ If $A_R^{\diamond}\neq \emptyset,$ then $A^c\in \mathcal{P}.$  
\end{theorem}

\begin{proof}
Let $A_R^{\diamond}\neq \emptyset.$
\\
$
\left.\begin{array}{r}
A_R^{\diamond}\neq \emptyset\Rightarrow (\exists x\in X)(x\in A_R^{\diamond})\Rightarrow (\forall U\in RO(X,x))(A^c\subseteq U^c\cup A^c\in\mathcal{P}) \\ \mathcal{P} \text{ is a primal on } X
\end{array}\right\}\Rightarrow A^c\in\mathcal{P}.
$
\end{proof}



\begin{corollary}\label{20}
Let $\mathcal{P}$ be a primal on topological space $(X,\tau)$ and $A\subseteq X.$ If $A^c\notin \mathcal{P},$ then $A_R^{\diamond}=\emptyset.$
\end{corollary}

\begin{proof}
It is obvious from Theorem \ref{5}.
\end{proof}

\begin{theorem} \label{4}
Let $(X,\tau,\mathcal{P})$ be a primal topological space and $A,B\subseteq X.$ Then, $A_R^{\diamond}\setminus B_R^{\diamond}=(A\setminus B)_R^{\diamond}\setminus B_R^{\diamond}.$ \end{theorem}

\begin{proof}
Let $A,B\subseteq X.$
\\

$\begin{array}{rcl}
A,B\subseteq X & \Rightarrow & A=(A\setminus B)\cup (A\cap B) \\ \\ & \Rightarrow & A_R^{\diamond}  =  [(A\setminus B)\cup (A\cap B)]_R^{\diamond} \\ \\
& \Rightarrow & A_R^{\diamond} = (A\setminus B)_R^{\diamond}\cup (A\cap B)_R^{\diamond} \subseteq   (A\setminus B)_R^{\diamond}\cup B_R^{\diamond}\\ \\
&\Rightarrow& A_R^{\diamond}\setminus B_R^{\diamond}\subseteq (A\setminus B)_R^{\diamond}  \\ \\ & \Rightarrow & A_R^{\diamond}\setminus B_R^{\diamond}\subseteq (A\setminus B)_R^{\diamond}\setminus B_R^{\diamond}\ldots (1)
\end{array}
$
\\

$
\begin{array}{rcl}
A,B\subseteq X & \Rightarrow & A\setminus B\subseteq A \\ \\ & \Rightarrow & (A\setminus B)_R^{\diamond}\subseteq A_R^{\diamond} \\ \\
&\Rightarrow & (A\setminus B)_R^{\diamond}\setminus B_R^{\diamond}\subseteq A_R^{\diamond}\setminus B_R^{\diamond}\ldots (2)
\end{array}$

$\begin{array}{l}
(1),(2)\Rightarrow A_R^{\diamond}\setminus B_R^{\diamond}=(A\setminus B)_R^{\diamond}\setminus B_R^{\diamond}.
\end{array}$
\end{proof}

\begin{theorem}
Let $(X,\tau,\mathcal{P})$ be a primal topological space and $A,B\subseteq X.$ If $B^c\notin\mathcal{P},$ then $(A\cup B)_R^{\diamond}=A_R^{\diamond}=(A\setminus B)_R^{\diamond}.$
\end{theorem}

\begin{proof}
Let $A,B\subseteq X.$
\\
$
\left.\begin{array}{r}
A,B\subseteq X\overset{\text{Theorem } \ref{4}}{\Rightarrow} A_R^{\diamond}\setminus B_R^{\diamond}=(A\setminus B)_R^{\diamond}\setminus B_R^{\diamond}
\\
B^c\notin\mathcal{P}\overset{\text{Corollary }\ref{20}}{\Rightarrow} B_R^{\diamond}=\emptyset
\end{array}\right\}\Rightarrow A_R^{\diamond}=(A\setminus B)_R^{\diamond}\ldots (1)
$
\\
$
\left.\begin{array}{r}
A,B\subseteq X\overset{\text{Theorem } \ref{1}}{\Rightarrow} (A\cup B)_R^{\diamond}=A_R^{\diamond}\cup B_R^{\diamond}
\\
B^c\notin\mathcal{P}\overset{\text{Corollary }\ref{20}}{\Rightarrow}  B_R^{\diamond}=\emptyset
\end{array}\right\}\Rightarrow (A\cup B)_R^{\diamond}=A_R^{\diamond}\ldots (2)
$
\\

$
\begin{array}{l}
(1),(2)\Rightarrow (A\cup B)_R^{\diamond}=A_R^{\diamond}=(A\setminus B)_R^{\diamond}.
\end{array}
$
\end{proof}

\section{The $cl^{\diamond}_R$ operator and $\tau^{\diamond}_R$ topology}

In this section, we define and investigate a new operator by means of the $(\cdot)^{\diamond}_R$ operator. The operator defined in this section comes across as a Kuratowski closure operator while the operator we defined in the previous section is not. In this way, we obtain a new topology, called $\tau^{\diamond}_R,$ which is finer than $\tau_{\delta}.$ Also, we get a base for this new topology.

\begin{definition}\label{1}
Let $(X,\tau,\mathcal{P})$ be a primal topological space. We define an operator $cl^{\diamond}_R: 2^X\rightarrow 2^X$ as $cl^{\diamond}_R(A)=A\cup A^{\diamond}_R$, where $A$ is any subset of $X$. 
\end{definition}

\begin{corollary} 
Let $X$ be a topological space and $A\subseteq X.$ If $X$ is a regular space, then the sets $cl_R^{\diamond}(A)$ and $cl^{\diamond}(A)$ coincide. 
\end{corollary}

\begin{proof}
It is obvious from Corollary \ref{15} and Definition \ref{1}. 
\end{proof}

\begin{corollary} \label{17}
Let $X$ be a topological space and $A\subseteq X.$ If $X$ is a compact Hausdorff space, then the sets $cl_R^{\diamond}(A)$ and $cl^{\diamond}(A)$ coincide.    
\end{corollary}

\begin{proof}
It is obvious from Corollary \ref{16}.
\end{proof}

\begin{theorem} \label{2}
Let $(X,\tau,\mathcal{P})$ be a primal topological space and $A,B\subseteq X.$ Then, the following statements hold: 
\\

$a)$ $cl^{\diamond}_R(\emptyset)= \emptyset$,\\

$b)$ $cl^{\diamond}_R(X)= X$,\\

$c)$ $A\subseteq cl^{\diamond}(A)\subseteq cl^{\diamond}_R(A)$,\\

$d)$ If $A\subseteq B$, then $cl^{\diamond}_R(A)\subseteq cl^{\diamond}_R(B)$,\\

$e)$ $cl^{\diamond}_R(A\cup B)=cl^{\diamond}_R(A)\cup cl^{\diamond}_R(B),$\\

$f)$ $cl^{\diamond}_R(cl^{\diamond}_R(A))=cl^{\diamond}_R(A).$

\end{theorem}

\begin{proof}
$a)$ Since $\emptyset^{\diamond}_R=\emptyset,$ we have $cl^{\diamond}_R(\emptyset)=\emptyset\cup \emptyset^{\diamond}_R=\emptyset.$  
\\

$b)$ Since $X^{\diamond}_R\subseteq X,$ we have $cl^{\diamond}_R(X)=X\cup X^{\diamond}_R=X.$\\

$c)$ Since $cl^{\diamond}_R(A)=A\cup A^{\diamond}_R,$ we have $A\subseteq cl^{\diamond}_R(A).$ Also, since $A^{\diamond}\subseteq A^{\diamond}_R,$ we have $cl^{\diamond}(A)\subseteq cl^{\diamond}_R(A).$\\

$d)$ Let $A\subseteq B.$
$$A\subseteq B\Rightarrow A^{\diamond}_R\subseteq B^{\diamond}_R\Rightarrow A\cup  A^{\diamond}_R\subseteq B\cup  B^{\diamond}_R\Rightarrow cl_R^{\diamond}(A)\subseteq cl_R^{\diamond}(B).$$

$e)$ Let $A,B\subseteq X.$
$$\begin{array}{rcl}
cl^{\diamond}_R(A\cup B) & = & (A\cup B)\cup (A\cup B)_R^{\diamond} \\ \\ & = & (A\cup B)\cup (A_R^{\diamond}\cup B_R^{\diamond}) \\ \\ & = & (A\cup A_R^{\diamond})\cup (B\cup B_R^{\diamond}) \\ \\ & = & cl^{\diamond}_R(A)\cup cl^{\diamond}_R(B).
\end{array}$$
\\

$f)$ Let $A\subseteq X.$ It is obvious from $(c)$ that $cl^{\diamond}_R(A)\subseteq cl^{\diamond}_R(cl^{\diamond}_R(A))\ldots (1)$
\\

$\left.\begin{array}{r}
cl^{\diamond}_R(cl^{\diamond}_R(A))  =  cl^{\diamond}_R(A)\cup \left(cl^{\diamond}_R(A)\right)^{\diamond}_R =  cl^{\diamond}_R(A)\cup (A\cup A^{\diamond}_R)^{\diamond}_R = cl^{\diamond}_R(A)\cup A^{\diamond}_R\cup (A^{\diamond}_R)^{\diamond}_R \\
A\subseteq X\Rightarrow A^{\diamond}_R\in\delta C(X)\Rightarrow
(A^{\diamond}_R)^{\diamond}_R\subseteq A^{\diamond}_R\end{array}\right\}\Rightarrow $
\\
$
\begin{array}{l}
\Rightarrow cl^{\diamond}_R(cl^{\diamond}_R(A))  \subseteq cl^{\diamond}_R(A)\cup A^{\diamond}_R(A)\cup A^{\diamond}_R(A) =  cl^{\diamond}_R(A)\ldots (2)
\end{array}
$
\\

$
\begin{array}{l}
(1),(2)\Rightarrow cl^{\diamond}_R(cl^{\diamond}_R(A))  = cl^{\diamond}_R(A).
\end{array}
$
\end{proof}

\begin{corollary}
Let $(X,\tau,\mathcal{P})$ be a primal topological space. Then the operator $cl^{\diamond}_R: 2^X\rightarrow 2^X$ defined by $cl^{\diamond}_R(A)=A\cup A^{\diamond}_R$, where $A$ is any subset of $X,$ is a Kuratowski's closure operator.
\end{corollary}

\begin{definition}
Let $(X,\tau,\mathcal{P})$ be a primal topological space. Then, the family $\tau^{\diamond}_R=\{A\subseteq X: cl^{\diamond}_R(A^c)=A^c\}$ is a topology on $X$ induced by topology $\tau$ and primal $\mathcal{P}.$ We can also use the notation  $\tau^{\diamond}_{R(\mathcal{P})}$ instead of $\tau^{\diamond}_R$ to indicate the primal as per our requirements.
\end{definition}

\begin{corollary}
Let $X$ be a topological space and $A\subseteq X.$ If $X$ is a regular space, then the topologies $\tau_R^{\diamond}$ and $\tau^{\diamond}$ coincide.
\end{corollary}

\begin{proof}
It is obvious from Corollary \ref{17}.
\end{proof}

\begin{corollary}
Let $X$ be a topological space and $A\subseteq X.$ If $X$ is a compact Hausdorff space, then the topologies $\tau_R^{\diamond}$ and $\tau^{\diamond}$ coincide.  
\end{corollary}

\begin{proof}
It is obvious since every compact Hausdorff space is regular.
\end{proof}

\begin{theorem} \label{tausubset}
Let $(X,\tau,\mathcal{P})$ be a primal topological space. Then, the following statements hold:
\\

$a)$ $\tau_{\delta}\subseteq \tau^{\diamond}_R,$ where $\tau_{\delta}$ is the family of all $\delta$-open sets in a topological space $(X,\tau).$
\\ 

$b)$ $\tau^{\diamond}_R\subseteq \tau^{\diamond}.$
\end{theorem}

\begin{proof}
$a)$ Let $A\in\tau_{\delta}.$ Our aim is to show that $A\in \tau^{\diamond}_R.$
\\

$
\left.\begin{array}{r}
A\in\tau_{\delta}\Rightarrow A^c\in \delta C(X)\Rightarrow (A^c)^{\diamond}_R\subseteq A^c\Rightarrow  A^c\cup(A^c)^{\diamond}_R=A^c \\ cl^{\diamond}_R(A^c)=A^c\cup (A^c)^{\diamond}_R
\end{array}\right\}\Rightarrow cl^{\diamond}_R(A^c)=A^c\Rightarrow A\in
\tau^{\diamond}_R.
$
\\

$b)$ Let $A\in \tau^{\diamond}_R.$ Our aim is to show that $A\in \tau^{\diamond}.$
$$
\begin{array}{rcl}
A\in \tau^{\diamond}_R & \Rightarrow & cl_R^{\diamond}(A^c)=A^c \\ \\
 & \Rightarrow & A^c\cup (A^c)_R^{\diamond}=A^c
 \\
 \\
 & \overset{\text{Theorem }\ref{1}}{\Rightarrow} & (A^c)^{\diamond}\subseteq(A^c)_R^{\diamond}\subseteq A^c
 \\
 \\
& \Rightarrow & (A^c)^{\diamond}\cup A^c =A^c
\\
\\
& \Rightarrow & cl^{\diamond}(A^c)=A^c
\\
\\
& \Rightarrow & A\in\tau^{\diamond}.\qedhere
\end{array}
$$

\end{proof}

\begin{corollary}
We have the following diagram from Theorem \ref{tausubset}.
\\
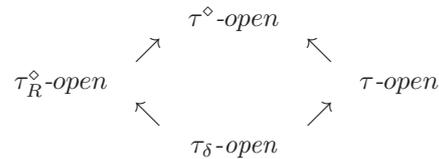
$$\begin{array}{ccccc}
 &  & \tau^{\diamond}\text{-}open &  &  \\
 & \nearrow &  & \nwarrow &    \\
\tau_R^{\diamond}\text{-}open &  &  &  & \tau\text{-}open \\
 & \nwarrow &  & \nearrow &  \\ 
 &  & \tau_{\delta}\text{-open} &  & 
\end{array}$$
\end{corollary}

\begin{remark}
The converses of the implications given in above diagram need not be true as shown by the following examples. Also, the notions of $\tau_R^{\diamond}$-open and $\tau$-open are independent with each other.
\end{remark}

\begin{example}
Let $X=\{a,b,c\}$ with the topology $\tau=\{\emptyset,X,\{a,b\},\{b,c\},\{b\}\}$ and let $\mathcal{P}=\{\emptyset,\{a\},\{b\},\{a,b\}\}.$ Simple calculations show that $\tau_{\delta}=\{\emptyset,X\},$  $\tau_R^{\diamond}=\{\emptyset,X,\{c\},\{a,c\},\{b,c\}\},$ and $\tau^{\diamond}=2^X.$ 
\\

$1)$ The set $\{c\}$ is $\tau_R^{\diamond}$-open but not $\tau_{\delta}$-open. 
\\

$2)$ The set $\{a\}$ is $\tau^{\diamond}$-open but not $\tau_R^{\diamond}$-open.
\\

$3)$ The set $\{c\}$ is $\tau_R^{\diamond}$-open but not $\tau$-open.
\\

$4)$ The set $\{b\}$ is $\tau$-open but not $\tau_R^{\diamond}$-open.
\end{example}

\begin{theorem}\label{3}
Let $(X,\tau,\mathcal{P})$ be a primal topological space and $A\subseteq X.$ Then, the following statements hold:\\

$a)$ $A\in\tau^{\diamond}_R$ if and only if for all $x$ in $A,$ there exists a regular open set $U$ containing $x$ such that $U^c\cup A\notin\mathcal{P},$\\

$b)$ if $A\notin \mathcal{P},$ then $A\in\tau^{\diamond}_R.$
\end{theorem}

\begin{proof} $a)$ Let $A\in\tau^{\diamond}_R.$
$$\begin{array}{rcl}A\in\tau^{\diamond}_R& \Leftrightarrow & cl^{\diamond}_R(A^c)=A^c \\ \\
& \Leftrightarrow & A^c\cup (A^c)^{\diamond}_R=A^c
\\ \\
& \Leftrightarrow & (A^c)^{\diamond}_R\subseteq A^c
\\ \\
& \Leftrightarrow & A\subseteq ((A^c)^{\diamond}_R)^c
\\ \\
& \Leftrightarrow & (\forall x\in A)(x\notin (A^c)^{\diamond}_R)
\\ \\
& \Leftrightarrow & (\forall x\in  A)(\exists U\in RO(X,x))(U^c\cup (A^c)^c=U^c\cup A\notin\mathcal{P}).
\end{array}$$

$b)$ Let $A\notin\mathcal{P}$ and $x\in A.$ We will make use of $(a).$
\\

$\left.\begin{array}{r}
(U:=X)(x\in A)\Rightarrow (U\in RO(X,x))(A=U^c\cup A)  \\ 
 A\notin\mathcal{P} 
\end{array}\right\}\Rightarrow U^c\cup A \notin\mathcal{P}.$
\end{proof}

\begin{remark}
 The converse of Theorem \ref{3}(b) need not be always true as shown by the following example.   
\end{remark}

\begin{example}
Let $X=\{a,b,c\}$ with the topology $\tau=\{\emptyset,X,\{a\},\{b\},\{a,b\}\}$ and $\mathcal{P}=\{\emptyset,\{a\},\{b\},\{a,b\}\}.$ Simple calculations show that $\tau^{\diamond}_R=2^X.$ It is obvious that the set $\{b\}$ belongs to both $\tau_R^{\diamond}$ and $\mathcal{P}.$ 
\end{example}

\begin{theorem}\label{karsit}
Let $(X,\tau,\mathcal{P})$ be a primal topological space. Then, the following statements hold:\\

$a)$ if $\mathcal{P}=\emptyset,$ then $\tau^{\diamond}_R=2^X,$\\

$b)$ if $\mathcal{P}=2^X\setminus\{X\},$ then $\tau_{\delta}=\tau^{\diamond}_R.$\\
\end{theorem}

\begin{proof}
$a)$ We have always $\tau_R^{\diamond}\subseteq 2^X\ldots (1).$ Now, let $A\in 2^X.$ Our aim is to show that $A\in \tau_R^{\diamond}.$ 
\\
$
\left.\begin{array}{r}
cl^{\diamond}_R(A^c)=A^c\cup (A^c)^{\diamond}_R  \\
\mathcal{P}=\emptyset \Rightarrow (A^c)_R^{\diamond}=\emptyset
\end{array}\right\}\Rightarrow cl^{\diamond}_R(A^c)=A^c\Rightarrow A\in
\tau^{\diamond}_R
$
\\

Then, we have $2^X\subseteq \tau^{\diamond}_R\ldots (2)$
\\

$(1),(2)\Rightarrow \tau^{\diamond}_R=2^X$.
\\


$b)$ We have always $\tau_{\delta}\subseteq\tau^{\diamond}_R$ from Theorem \ref{tausubset}.  Now, we will prove that $\tau^{\diamond}_R\subseteq\tau_{\delta}.$ Let $A\in\tau^{\diamond}_R.$
\\
$
\left.\begin{array}{r}
   A\in\tau_R^{\diamond} \overset{\text{Theorem }\ref{3}} {\Rightarrow} (\forall x\in A)(\exists U\in RO(X,x))(U^c\cup A\notin \mathcal{P})  \\ \mathcal{P}=2^X\setminus \{X\}
\end{array}\right\}\Rightarrow 
$
\\
$
\begin{array}{l}
\Rightarrow (\forall x\in A)(\exists U\in RO(X,x))(U^c\cup A=X)
\end{array}
$
\\
$
\begin{array}{l}
\Rightarrow (\forall x\in A)(\exists U\in RO(X,x))(U\cap A^c=\emptyset)
\end{array}
$
\\
$
\begin{array}{l}
\Rightarrow (\forall x\in A)(x\notin \delta\text{-}cl(A^c)=(\delta\text{-}int(A))^c)
\end{array}
$
\\
$
\begin{array}{l}
\Rightarrow (\forall x\in A)(x\in \delta\text{-}int(A))
\end{array}
$
\\
$
\begin{array}{l}
\Rightarrow A=\delta\text{-}int(A)
\end{array}
$
\\
$
\begin{array}{l}
\Rightarrow A\in\tau_{\delta}.
\end{array}
$
\end{proof}

\begin{remark}
The converses of Theorem \ref{karsit}(a) and Theorem \ref{karsit}(b) need not be true as shown by the following examples. 
\end{remark}

\begin{example}
Let $X=\{a,b,c\}$ with the topology $\tau=\{\emptyset,X,\{a\},\{b\},\{a,b\}\}$ and $\mathcal{P}=\{\emptyset,\{a\},\{b\},\{a,b\}\}.$ Simple calculations show that $\tau^{\diamond}_R=2^X$, but $\mathcal{P}\neq \emptyset.$
\end{example}

\begin{example}
Let $X=\{a,b,c\}$ with the topology $\tau=\{\emptyset,X,\{a\},\{b\},\{a,b\}\}$ and $\mathcal{P}=2^X\setminus \{X,\{a,b\}\}.$ Simple calculations show that $\tau=\tau_{\delta}=\tau^{\diamond}_R$, but $\mathcal{P}\neq 2^X\setminus\{X\}.$
\end{example}

\begin{theorem} \label{2}
Let $(X,\tau,\mathcal{P})$ be a primal topological space. Then, the family $\mathcal{B}=\{T\cap P|(T\in\tau_{\delta})(P\notin \mathcal{P})\}$ is a base for the topology $\tau^{\diamond}_R$ on $X.$
\end{theorem}

\begin{proof}
Let $B\in\mathcal{B}.$ 
\\
$
\left.\begin{array}{r}
B\in\mathcal{B}\Rightarrow (\exists T\in\tau_{\delta})(\exists P\notin\mathcal{P})(B=T\cap P)
\\
\tau_{\delta}\subseteq \tau_R^{\diamond}
\end{array}\right\}\overset{\text{Theorem }\ref{3}}{\Rightarrow} (T,P\in\tau_R^{\diamond})(B=T\cap P)\Rightarrow B\in \tau_R^{\diamond}
$
\\

Then, we have $\mathcal{B}\subseteq \tau_R^{\diamond}.$ Now, let $A\in\tau_R^{\diamond}$ and $x\in A.$ Our aim is to find $B\in\mathcal{B}$ such that $x\in B\subseteq A.$
\\
$
\left.\begin{array}{r}
x\in A\in \tau^{\diamond}_R\overset{\text{Theorem }\ref{3}}{\Rightarrow}(\exists U\in RO(X,x))(U^c\cup A\notin\mathcal{P})
\\
B:=U\cap (U^c\cup A)
\end{array}\right\}\Rightarrow (B\in\mathcal{B})(x\in B\subseteq A)
$

Hence, $\mathcal{B}$ is a base for the topology  $\tau_R^{\diamond}$ on $X.$
\end{proof}


\begin{theorem}
Let $(X,\tau,\mathcal{P})$ and $(X,\tau,\mathcal{Q})$ be two primal topological spaces. If $\mathcal{P}\subseteq \mathcal{Q},$ then $\tau_{R(\mathcal{Q})}^{\diamond}\subseteq \tau_{R(\mathcal{P})}^{\diamond}.$
\end{theorem}

\begin{proof}
Let $A\in \tau_{R(\mathcal{Q})}^{\diamond}.$
\\
$\left.\begin{array}{r}
A\in \tau_{R(\mathcal{Q})}^{\diamond}\Rightarrow (\forall x\in A)(\exists U\in RO(X,x))(U^c\cup A\notin\mathcal{Q})
\\
\mathcal{P}\subseteq \mathcal{Q}
\end{array}
\right\}\Rightarrow $
\\
$
\begin{array}{l}
\Rightarrow (\forall x\in A)(\exists U\in RO(X,x))(U^c\cup A\notin\mathcal{P})
\end{array}
$
\\
$
\begin{array}{l}
\Rightarrow A\in \tau_{R(\mathcal{P})}^{\diamond}.
\end{array}
$
\end{proof}


\section{Conclusion}
In this study, we defined and studied two new operators showed with $(\cdot)_R^{\diamond}$ and $cl_R^{\diamond}(\cdot)$ via the notion of primal. While the first one is not a Kuratowski closure operator, the second one comes across as a Kuratowski closure operator. Hence, we obtained a new topology $\tau_R^{\diamond}$ that is finer than $\tau_{\delta}.$ On the other hand,  we showed that the notions of $\tau_R^{\diamond}$-open and $\tau$-open are independent. Moreover, we obtained a base for this new topology $\tau_R^{\diamond}$ and proved several fundamental results. Furthermore, we give not only some relationships but also several examples. We believe that this study will help researchers to upgrade and support further studies on primals.   
\\

{\bf Conflict of interest:} The authors declare that there is no conflict of interest.

\end{document}